\documentclass[reqno,12pt]{amsart}

\usepackage{amsmath, epsfig, cite}
\usepackage{amssymb}
\usepackage{amsfonts}
\usepackage{latexsym}
\usepackage{amsthm}

\newtheorem{theorem}{Theorem}

\newtheorem{corollary}[theorem]{Corollary}
\newtheorem{conjecture}{Conjecture}
\newtheorem{lemma}{Lemma}

\theoremstyle{remark}

\numberwithin{equation}{section}

\allowdisplaybreaks

\author{Yudong Liu and Xiaoxia Wang$^*$}
\address{Department of Mathematics\\
    Shanghai University \\
    Shanghai 200444, P.\:R.\:China}
\email{lydshdx@163.com (Y. Liu), xiaoxiawang@shu.edu.cn (X. Wang)}
\thanks{This work is supported by National Natural Science Foundations of China (11661032).}
\thanks{$*$ Corresopnding author}
\title[Further $q$-analogues of the  (G.2) Supercongruence of Van Hamme]{Further $q$-analogues of the  (G.2) Supercongruence of Van Hamme}



\begin{document}

\begin{abstract}
In 2015, Swisher generalized the (G.2)  supercongruence of Van Hamme to the modulus $p^4$.
In this paper, we first propose two $q$-analogues of Swisher's supercongruence and then a new $q$-congruence with parameters 
is present.
Furthermore, we prove a $q$-congruence modulo the fourth power of a cyclotomic polynomial, which was conjectured by the authors early.
\end{abstract}

\maketitle

\section{Introduction}
In 1997, Van Hamme\cite{Hamme} developed 13 mysterious $p$-adic analogues of Ramanujan-type $\pi$-formulas which were marked as (A.2)--(M.2), such as,
\begin{align}
(\text{G.2})\quad\quad&\sum_{k=0}^{(p-1)/4}(8k+1)\frac{(\frac{1}{4})_k^4}{k!^4}
\equiv p\frac{\Gamma_p(\frac 12)\Gamma_p(\frac 14)}{\Gamma_p(\frac 34)}
\pmod{p^3},  \quad\quad \text{if $p\equiv 1 \pmod{4}$}.
\label{eq:pram}
\end{align}
Here and throughout this paper, $p$ is an odd prime and $\Gamma_p(x)$ is the $p$-adic Gamma
function~\cite{Mor}.
Then the research of the Ramanujan-type congruence and supercongruence has caught a lot of authors' attention. Later, Swisher \cite[Theorem 1.2]{Swisher} proved a stronger version of \eqref{eq:pram}:
\begin{align}
\sum_{k=0}^{(p-1)/4}(8k+1)\frac{(\frac{1}{4})_k^4}{k!^4}
\equiv -(-1)^{\frac{p-1}{4}}p\Gamma_p(\tfrac 12)\Gamma_p(\tfrac 14)^2
\pmod{p^4},  \quad\quad\text{if $p\equiv 1 \pmod{4}$}.
\label{eq:swisher1}
\end{align}
Since $\Gamma_p(\frac 14)\Gamma_p(\frac 34)=-(-1)^{\frac{p-1}{4}}$ for $p\equiv 1 \pmod{4}$, the right-hand side of \eqref{eq:swisher1} is the same as that of \eqref{eq:pram}.

Recently, the authors\cite{LW} gave two $q$-analogues of \eqref{eq:pram} as follows: for positive integers $n\equiv 1 \pmod 4$,
\begin{align}
&\sum_{k=0}^{M}[8k+1] \frac{\left(q ; q^{4}\right)_k^4}{\left(q^4 ; q^{4}\right)_k^4}q^{2k} \equiv \frac{\left(q^{2} ; q^{4}\right)_{(n-1)/{4}}}{\left(q^{4} ; q^{4}\right)_{(n-1)/4}}[n] q^{(1-n) / 4}
\pmod {[n]\Phi_{n}(q)^{2}}\label{eq;1};\\
&\sum_{k=0}^{M}[8k+1]_{q^{2}}[8k+1]^{2} \frac{\left(q^{2} ; q^{8}\right)_k^{4}}{\left(q^{8} ; q^{8}\right)_k^{4}} q^{-4k }
\equiv -\frac{2[n]_{q^{2}}(q^{4};q^{8})_{(n-1)/4}}{(1+q^{2})(q^{8};q^{8})_{(n-1)/4}}q^{(3-n)/2} \pmod{[n]_{q^2} \Phi_{n}(q^2)^2} ,\label{eq;2}
\end{align}
here and in what follows $M=(n-1)/4$ or $n-1$.

In fact, Guo and Schlosser \cite[Theorem 2]{GS3} have presented a $q$-supercongruence: for even $d\geq4$ and integer $n\geq1$ with $n\equiv -1 \pmod d$,
\begin{align}
\sum_{k=0}^{n-1}[2dk+1]\frac{(q;q^d)_{k}^d}{(q^d;q^d)_{k}^d}q^{\frac{d(d-3)k}{2}}\equiv 0 \pmod {\Phi_{n}(q)^2},
\end{align}
which is just $q$-analogue of a companion of Van Hamme's (G.2) for $p\equiv 3 \pmod {4}$ when $d=4$.
The authors\cite{LW1} also proposed the following conjecture.
\begin{conjecture}\label{con1}
Let $n\equiv 1\pmod {4}$ be a positive integer. Then, modulo $ [n] \Phi_{n}(q)^{3}$,
\begin{align}
\sum_{k=0}^{(n-1)/2}[6 k+1] \frac{(q ; q^{4})_{k}(q ; q^{2})_{k}^{3}}{(q^{2} ; q^{2})_{k}(q^{4} ; q^{4})_{k}^{3}} q^{k^{2}+k}
\equiv \sum_{k=0}^{(n-1)/4}[8k+1] \frac{(q ; q^{4})_{k}^4}{(q^4 ; q^{4})_k^4}q^{2k}.
\end{align}
\end{conjecture}
Here and in what follows, $(a;q)_n=(1-a)(1-aq)\cdots (1-aq^{n-1})$
is  the {\em $q$-shifted factorial} and
$(a_1,a_2,\ldots,a_m;q)_n=(a_1;q)_n (a_2;q)_n\cdots (a_m;q)_n$
denotes the product of $q$-shifted factorials for simplicity.
{\em$q$-integer} $[n]$ is defined as $[n]=[n]_q=(1-q^n)/(1-q)$ and
$\Phi_n(q)$ is the $n$-th {\em cyclotomic polynomial}.

For some other recent progress on congruences and $q$-congruences, see
\cite{Guo5,Guo6,Guo-diff,Guo7,GS20,GS2,Guo21,GuoZu,WP,WY0,WY,Liu,LP,NP,NP2,Sun,Zudilin,WY1}.
Especially, Guo and Zudlin \cite{GuoZu} introduced the `creative microscoping' method which is useful for proving $q$-congruences.

The first aim of this paper is to give  the following two $q$-supercongruences modulo the fourth power
of a cyclotomic polynomial which are the $q$-analogues of \eqref{eq:swisher1} and also the generalizations of the $q$-congruences \eqref{eq;1} and \eqref{eq;2}.
\begin{theorem}\label{thm;e}
Let $n\equiv 1 \pmod4$ be positive integer. Then, modulo $[n]\Phi_n(q)^3$, we have
\begin{align}
&\sum_{k=0}^{M}[8 k+1] \frac{(q ; q^{4})_k^4}{(q^{4} ; q^{4})_k^{4}} q^{2k}
 \equiv\frac{\left(q^{2 } ; q^{4}\right)_{(n-1) / 4}}{\left(q^{4} ; q^{4}\right)_{(n-1) / 4}}[n] q^{(1-n) / 4}\big\{1+[n]^2B(n,q)\big\};\label{e1}\\
&\sum_{k=0}^{M}[8k+1]_{q^{2}}[8k+1]^{2} \frac{\left(q^{2} ; q^{8}\right)_{k}^{4}}{\left(q^{8} ; q^{8}\right)_{k}^{4}} q^{-4k }\nonumber\\
&\quad\quad\quad \quad\quad\equiv -\frac{2[n]_{q^{2}}(q^{4};q^{8})_{(n-1)/4}}{(1+q^{2})(q^{8};q^{8})_{(n-1)/4}}q^{(5-n)/2}\big\{1+[n]_{q^2}^2B(n,q^2)\big\},\label{5}
\end{align}
here and throughout this paper \[B(n,q)=\frac{(n^2-1)(1-q)^2}{24}+\sum_{k=1}^{(n-1)/4}\frac{q^{4k-2}}{[4k-2]^2}.\]
\end{theorem}


The second purpose of this paper is to prove Conjecture \ref{con1} by comparing the following result with the $q$-supercongruence \eqref{e1}.
\begin{theorem}\label{thm;4}
Let $n\equiv 1\pmod {4}$ be a positive integer. Then, modulo $ [n] \Phi_{n}(q)^{3}$,
\begin{align}
\sum_{k=0}^{(n-1) / 2}[6 k+1] \frac{(q ; q^{4})_{k}(q ; q^{2})_{k}^{3}}{\left(q^{2} ; q^{2}\right)_{k}(q^{4} ; q^{4})_{k}^{3}} q^{k^{2}+k}  &\equiv
\frac{\left(q^{2} ; q^{4}\right)_{(n-1)/4}}{\left(q^{4} ; q^{4}\right)_{(n-1)/4}}[n] q^{(1-n) / 4}\big\{1+[n]^2B(n,q)\big\}.
\end{align}
\end{theorem}

The rest of the paper is organized as follows.
We shall prove Theorem \ref{thm;e} by establishing a parameter generalized $q$-congruence which based on Watson's $_8\phi_7$ transformation and the Chinese remainder theorem for coprime polynomials in the next section.
In Section 3, we will give simple proof of Theorem \ref{thm;4}.
Finally, we will give another five $q$-analogues of \eqref{eq:swisher1} in Section 4.

\section{Proof of Theorem \ref{thm;e} }\label{sec:thm1}
The following  Watson's $_8\phi_7$ transformation (cf.\cite[Appendix~(II.18)]{GR})
\begin{align}
& _{8}\phi_{7}\!\left[\begin{array}{cccccccc}
a,& qa^{\frac{1}{2}},& -qa^{\frac{1}{2}}, & b,    & c,    & d,    & e,    & q^{-N} \\
& a^{\frac{1}{2}}, & -a^{\frac{1}{2}},  & aq/b, & aq/c, & aq/d, & aq/e, & aq^{N+1}
\end{array};q,\, \frac{a^2q^{N+2}}{bcde}
\right] \notag\\[5pt]
&\quad =\frac{(aq, aq/de;q)_N}
{(aq/d, aq/e;q)_N}
\,{}_{4}\phi_{3}\!\left[\begin{array}{c}
aq/bc,\ d,\ e,\ q^{-N} \\
aq/b,\, aq/c,\, deq^{-N}/a
\end{array};q,\, q
\right]\label{eq:8phi7}
\end{align}
will play an important role in our proof.
In fact, we can prove Theorem \ref{thm;e} by establishing the following parameter generalized $q$-congruence. Obviously, 
the special cases $c=d=q^{3/2}$ and  $q\rightarrow q^2, c=d=q^{7/2}$ of the following $q$-congruence are just the two results in Theorem \ref{thm;e}.
\begin{theorem}\label{thm;1}
Let $n\equiv 1 \pmod4$ be positive integer. Then, modulo $[n]\Phi_n(q)^3$, we have
\begin{align}
&\sum_{k=0}^{M} [8 k+1]\frac{(q;q^4)_k^4(cq;q^4)_k(qd;q^4)_k}{(q^4;q^4)_k^4(q^4/c;q^4)_k(q^4/d;q^4)_k} ({cd})^{-k} q^{5 k}\nonumber\\
&\quad\equiv\frac{\left(q^{2 } ; q^{4}\right)_{(n-1) / 4}}{\left(q^{4} ; q^{4}\right)_{(n-1) / 4}}[n] q^{(1-n) / 4}
\big\{1+[n]^2B(n,q)\big\}\nonumber\\
&\quad\quad\times \sum_{k=0}^{M}\frac{(q^{3}/cd;q^4)_k(q;q^4)_k^3}{(q^4;q^4)_k(q^4/c;q^4)_k(q^4/d;q^4)_k(q^{2};q^4)_k}q^{4k}
.\label{eq;thm1_1}
\end{align}
\end{theorem}

In order to prove Theorem \ref{thm;1}, we first recall the following two parametric $q$-congruences which the authors \cite[Lemmas 1 and 2]{LW2} have proved.
\begin{lemma}\label{Lemma2}
Let $d, n$ be positive integers with $\gcd(d, n)=1$. Let $r$ be an integer and let $a$, $b$, $c$ and $e$ be indeterminates. Then, modulo $[n]$,
\begin{align}
&\sum_{k=0}^{m_1} [2 d k+r]\frac{(q^r,cq^r,eq^r,q^r/b,aq^{r}, q^{r}/a; q^{d})_{k}}{(q^d,q^d/c,q^d/e,q^db,q^{d}/a,aq^{d};q^{d})_{k}}
\left(\frac{b}{ce}\right)^k q^{(2d-3 r) k}\equiv 0\label{eq;Lemma1};\\
&\sum_{k=0}^{n-1} [2 d k+r]\frac{(q^r,cq^r,eq^r,q^r/b,aq^{r}, q^{r}/a; q^{d})_{k}}{(q^d,q^d/c,q^d/e,q^db,q^{d}/a,aq^{d};q^{d})_{k}}
\left(\frac{b}{ce}\right)^k q^{(2d-3 r) k}\equiv 0\label{eq;Lemma2},
\end{align}
where $0\le m_1 \le n-1$ and $dm_1 \equiv -r \pmod n$.
\end{lemma}

\begin{lemma}\label{lemma1}
Let $n>1$, $d\ge 2$ , $r$ be  integers with $\gcd(r,d)=1$ and $n\equiv r \pmod {d}$ such that $n+d-nd\le r\le n$. Then, modulo $\Phi_{n}(q)(1-aq^n)(a-q^n)$,
\begin{align}
&\sum_{k=0}^{(n-r)/d} [2 d k+r]\frac{(q^r, cq^r, eq^r, q^r/b, aq^{r}, q^{r}/a ; q^{d})_{k}}{(q^d, q^d/c, q^d/e, q^db, q^{d}/a, aq^{d}; q^{d})_{k}} (\frac{b}{ce})^k q^{(2d-3 r) k} \nonumber\\
&\quad\equiv\frac{\left(q^{2 r}/b ; q^{d}\right)_{(n-r) / d}}{\left(bq^{d} ; q^{d}\right)_{(n-r) / d}}[n] \big(\frac{b}{q^r}\big)^{(n-r) / d} \sum_{k=0}^{M}\frac{(q^{d-r}/ce, q^{r}/b, aq^{r},q^{r}/a;q^d)_k}{(q^d, q^d/c, q^d/e, q^{2r}/b;q^d)_k}q^{dk}.\label{eq;thm9_1}
\end{align}
\end{lemma}
We also need the following result.
\begin{lemma}\label{lemma2}
Let $n\equiv 1 \pmod4$ be positive integer. Then, modulo $b-q^{n}$,
\begin{align}
&\sum_{k=0}^{(n-1)/4} [8 k+1]\frac{(q, q/b, qc, qd, aq, q/a ; q^{4})_{k}}{(q^4, q^4b, q^4/c, q^4/d, q^{4}/a, aq^{4}; q^{4})_{k}}  (cd)^{-k}q^{5k}b^k \nonumber\\
&\quad\equiv[n]\frac{\left(q ,q^{3}; q^{4}\right)_{(n-1) / 4}}{\left(aq^{4},q^4/a ; q^{4}\right)_{(n-1) / 4}} \sum_{k=0}^{(n-1)/4}\frac{(q^{3}/cd, q/b, aq,q/a;q^4)_k}{(q^4,q^4/c,q^4/d,q^4,q^{2}/b;q^4)_k}q^{4k}. \label{eq;lemma3}
\end{align}
\end{lemma}
\begin{proof}
Letting  $q\rightarrow q^4$, $a=q$, $b=qc$, $c=qd$, $d=q/a$, $e=aq$ and $N=(n-1)/4$ in Watson's formula \eqref{eq:8phi7}, we get
\begin{align*}
&\sum_{k=0}^{(n-1)/4} [8 k+1]\frac{(q, q^{1-n}, qc, qd, aq, q/a; q^{4})_{k}}{(q^4, q^{4-n}, q^4/c, q^4/d, q^{4}/a, aq^{4}; q^{4})_{k}} (cd)^{-k} q^{(5+n)  k} \nonumber\\
&\quad\equiv[n]\frac{\left(q ,q^{3}; q^{4}\right)_{(n-1) / 4}}{\left(aq^{4},q^4/a ; q^{4}\right)_{(n-1) / 4}} \sum_{k=0}^{(n-1)/4}\frac{(q^{3}/cd, q^{1-n}, aq,q/a;q^4)_k}{(q^4,q^4/c,q^4/d,q^4,q^{2-n};q^4)_k}q^{4k},
\end{align*}
which is just the special case $b=q^n$ of \eqref{eq;lemma3}. This establishes Lemma \ref{lemma2}.
\end{proof}

We now give a parametric version of Theorem \ref{thm;1}.
\begin{theorem}\label{thm;5}
Let $n\equiv 1 \pmod4$ be positive integer. Then, modulo $\Phi_{n}(q)^2(a-q^n)(1-aq^n)$,
\begin{align}\label{eq;thm5}
&\sum_{k=0}^{(n-1)/4} [8 k+1]\frac{(q;q^4)_k^2(cq, dq, qa, q/a;q^4)_k}{(q^4;q^4)_k^2(q^4/c, q^4/d, q^4a, q^4/a;q^4)_k} ({cd})^{-k} q^{5 k} \nonumber\\
&\quad\equiv \left\{q^{(1-n) / 4}[n]\frac{\left(q^{2 }; q^{4}\right)_{(n-1) / 4}}{\left(q^{4} ; q^{4}\right)_{(n-1) / 4}} + q^{(1-n)/4}[n]\frac{h_q(a)}{(1-a^2)(1-a^n)(q^2,q^4;q^4)_{(n-1)/4}}\right\}\nonumber\\
&\quad\quad\times\sum_{k=0}^{(n-1)/4}\frac{(q, q^{3}/cd, aq, q/a;q^4)_k}{(q^{2},q^4,q^4/c,q^4/d;q^4)_k}q^{4k},
\end{align}
where
\[h_q(a)=(1-aq^n)(a-q^n)\left\{\left(q^{2 } ; q^{4}\right)_{(n-1) / 4}^2(1-a^n)-n(aq^2,q^2/a;q^4)_{(n-1)/4}(1-a)a^{(n-1)/2}\right\}.\]
\end{theorem}
\begin{proof}
It is clear that $\Phi_{n}(q)$, $1-aq^{n}$, $a-q^{n}$ and $b-q^{n}$ are pairwise relatively prime polynomials and the following congruences can be easily verified
\begin{equation}
\frac{\left(b-q^{n}\right)\left(a b-1-a^{2}+a q^{n}\right)}{(a-b)(1-a b)} \equiv 1 \quad\left(\bmod \left(1-a q^{n}\right)\left(a-q^{n}\right)\right);\label{eq1}
\end{equation}
\begin{equation}
\frac{\left(1-a q^{n}\right)\left(a-q^{n}\right)}{(a-b)(1-a b)} \equiv 1 \quad\left(\bmod \left(b-q^{n}\right)\right). \label{eq2}
\end{equation}
 Applying Lemma \ref{lemma1} for $d=4$, $r=1$, $e=d$ and Lemma \ref{lemma2} with the Chinese remainder theorem for coprime polynomials, we obtain, modulo $\Phi_{n}(q) (1-aq^{n}) (a-q^{n})  (b-q^{n})$,
\begin{align}
&\sum_{k=0}^{(n-1)/4} [8 k+1]\frac{(q, cq, dq, qa, q/a, q/b;q^4)_k}{(q^4, q^4/c, q^4/d, q^4a,q^4/a, q^4b;q^4)_k} ({cd})^{-k} q^{5 k}b^k \nonumber\\
&\quad\equiv [n]A_q(a,b,n)\sum_{k=0}^{(n-1)/4}\frac{(q^{3}/cd, q/b, aq, q/a;q^4)_k}{(q^4,q^4/c,q^4/d,q^4,q^{2}/b;q^4)_k}q^{4k},\label{eq;thm9_2}
\end{align}
where
\begin{align}
A_q(a,b,n)&=\frac{\left(b-q^{n}\right)\left(a b-1-a^{2}+a q^{n}\right)}{(a-b)(1-a b)} \frac{(b / q)^{(n-1) / 4}\left(q^{2} / b ; q^{4}\right)_{(n-1) / 4}}{\left(b q^{4} ; q^{4}\right)_{(n-1) / 4}}\nonumber\\
&\quad+\frac{\left(1-a q^{n}\right)\left(a-q^{n}\right)}{(a-b)(1-a b)} \frac{\left(q , q^3 ; q^{4}\right)_{(n-1) / 4}}{\left(a q^{4}, q^{4} / a ; q^{4}\right)_{(n-1) / 4}}.\label{eq;proof1}
\end{align}
It is easy to get
\begin{align}
\left(q^{4} / a; q^{4}\right)_{(n-1) / 4} &=\left(1-q^{4} / a\right)\left(1-q^{8} / a\right) \cdots\left(1-q^{n-1} / a\right) \nonumber\\
& \equiv\left(1-q^{4-n} / a\right)\left(1-q^{8-n} / a\right) \cdots\left(1-q^{-1} / a\right)\nonumber\\
&=(-1)^{(n-1)/4}(aq;q^4)_{(n-1)/4}\frac{q^{-(n-3)(n-1)/8}}{a^{(n-1)/4}}\pmod{\Phi_{n}(q)}\label{eq;proof2}.
\end{align}
Similarly, we have
\begin{align}
\left(aq^3 ; q^{4}\right)_{(n-1) / 4}  \equiv(-1)^{(n-1)/4}(q^2/a;q^4)_{(n-1)/4}q^{-(n-1)^2/8}a^{(n-1)/4}\pmod{\Phi_{n}(q)}\label{eq;proof3}.
\end{align}
Note that the following $q$-congruence which has been proposed by Guo\cite[Equation (2.3)]{Guo7}
\begin{align}
(a q ; q)_{n-1}\equiv \sum_{k=0}^{n-1} a^{k} \pmod {\Phi_{n}(q)}.
\end{align}
Then   $\left(a q^{4}, q^{4} / a ; q^{4}\right)_{(n-1) / 4}$ and $\left( q, q^{3}  ; q^{4}\right)_{(n-1) / 4}$ can be transformed as
\begin{align}
\left(a q^{4}, q^{4} / a ; q^{4}\right)_{(n-1) / 4}&\equiv (-1)^{(n-1)/4}\frac{q^{-(n-3)(n-1)/8}}{a^{(n-1)/4}}(aq,aq^4;q^4)_{(n-1)/4}\nonumber\\
&=(-1)^{(n-1)/4}\frac{q^{-(n-3)(n-1)/8}}{a^{(n-1)/4}}\frac{(aq;q)_{n-1}}{(aq^2,aq^3;q^4)_{(n-1)/4}}\nonumber\\
&\equiv q^{(1-n)/4}\frac{(1-a^n)}{(1-a)(aq^2,q^2/a;q^4)_{(n-1)/4}a^{(n-1)/2}}\pmod{\Phi_{n}(q)}\label{eq;proof5}
\end{align}
and
\begin{align}
\left( q, q^{3}  ; q^{4}\right)_{(n-1) / 4}&= \frac{(q;q)_{n-1}}{(q^2,q^4;q^4)_{(n-1)/4}}
\equiv \frac{n}{(q^2,q^4;q^4)_{(n-1)/4}} \pmod{\Phi_{n}(q)}\label{eq;proof6}.
\end{align}
Combine the $q$-congruences \eqref{eq;proof5} and \eqref{eq;proof6},  $A_q(a,b,n)$ reduces to the following result with $b=1$,
\begin{align}
A_q(a,1,n)&=q^{(1-n) / 4}\frac{\left(q^{2 } ; q^{4}\right)_{(n-1) / 4}}{\left(q^{4} ; q^{4}\right)_{(n-1) / 4}} + q^{(1-n)/4}\frac{h_q(a)}{(1-a)^2(1-a^n)(q^2,q^4;q^4)_{(n-1)/4}}.
\end{align}
Thus we arrive at Theorem \ref{thm;5} instantly after taking $b=1$ in both sides of \eqref{eq;thm9_2}.
\end{proof}
\begin{proof}[Proof of Theorem \ref{thm;1}]
We first need to compute the 3rd derivative of $h_q(a)$ at $a=1$ as
\begin{equation}
h_q^{\prime\prime\prime}(1)=(q^2;q^4)_{(n-1)/4}^2(1-q^n)^2\left\{-\frac{n(n^2-1)}{4}-6n\sum_{k=1}^{(n-1)/4}\frac{q^{4k-2}}{(1-q^{4k-2})^2}\right\}.
\end{equation}
According to  the L'Hospital rule, we have
\begin{align}
\lim _{a \rightarrow 1}\frac{h_q(a)}{(1-a^2)(1-a^n)(q^2,q^4;q^4)_{(n-1)/4}}
&=\frac{h_q^{\prime\prime\prime}(1)}{-6n(q^2,q^4;q^4)_{(n-1)/4}}\nonumber\\
&=\frac{\left(q^{2 } ; q^{4}\right)_{(n-1) / 4}}{\left(q^{4} ; q^{4}\right)_{(n-1) / 4}}[n]^2 B(n,q).\label{eq;proof7}
\end{align}
Owing to the fact that the denominators of both sides of \eqref{eq;thm5} as $a\rightarrow1$ are prime to $\Phi_n(q)^2$.
Combine with the result \eqref{eq;proof7}, the $q$-congruence \eqref{eq;thm1_1} for $M=(n-1)/4$  modulo $\Phi_{n}(q)^4$ is thus a direct conclusion from Theorem \ref{thm;5} by letting $a\rightarrow1$. On the other hand, since $(q;q^4)_k^4\equiv0 \pmod {\Phi_n(q)^4}$ for $k$ in the range $(n-1)/4\le k \le n-1$,  Theorem \ref{thm;1} is also true for $M=n-1$ modulo $\Phi_{n}(q)^4$. At the same time,  from Lemma \ref{Lemma2} for $d=4$, $r=a=b=1$ and $e=d$, the $q$-congruence \eqref{eq;thm1_1} is also valid modulo $[n]$. This completes the proof Theorem \ref{thm;1}.
\end{proof}

\section{Proof of Theorem \ref{thm;4} }

In order to prove Theorem \ref{thm;4}, we first need to establish the following $q$-congruence, which is the case of \cite[Theorem 4.5]{GuoZu} with $b\to q/b$ and $r=1$.
\begin{lemma}
Let $n\equiv 1 \pmod4$ be positive integer. Then, modulo $[n](a-q^n)(1-aq^n)$,
\begin{align}
\sum_{k=0}^{(n-1)/2} [6 k+1] \frac{(aq, q/a, q ; q^{2})_{k}\left(q/b ; q^{4}\right)_{k} q^{k^{2}+k}b^{k}}{(a q^{4},q^{4} / a, q^{4}; q^{4})_{k}\left(q^{2}  b ; q^{2}\right)_{k} }\equiv  \frac{\left(q^{2 }/b ; q^{4}\right)_{(n-1) / 4}}{\left(bq^{4} ; q^{4}\right)_{(n-1) / 4}}[n] (\frac{b}{q})^{(n-1) / 4}.\label{eq;proof9}
\end{align}
\end{lemma}
We also find a $q$-congruence for the left-hand side of \eqref{eq;proof9} modulo $b-q^n$.
\begin{lemma}
Let $n\equiv 1 \pmod4$ be positive integer. Then, modulo $b-q^n$,
\begin{align}
\sum_{k=0}^{(n-1)/2} [6 k+1]  \frac{(aq, q/a, q ; q^{2})_{k}\left(q/b ; q^{4}\right)_{k} q^{k^{2}+k}b^{k}}{(a q^{4},q^{4} / a, q^{4} ; q^{4})_{k}\left(q^{2}b ; q^{2}\right)_{k} }\equiv[n]\frac{\left(q ,q^{3}; q^{4}\right)_{(n-1) / 4}}{\left(aq^{4},q^4/a ; q^{4}\right)_{(n-1) / 4}}.\label{eq;proof10}
\end{align}
\end{lemma}
\begin{proof}[Sketch of proof.]
Using the following summation formula of Rahman\cite[eq. (4.6)]{Rahman}
\begin{align}
\sum_{k=0}^{\infty} \frac{\left(1-a q^{3 k}\right)(a, d, q / d ; q)_{k}\left(b ; q^{2}\right)_{k}}{(1-a)(q^{2}, a q^{2}/d, adq ; q^{2})_{k}(aq/b ; q)_{k}} \frac{a^{k} q^{k+1\choose 2}}{b^{k}} =\frac{\left(aq, aq^{2}, adq/b, aq^{2}/bd ; q^{2}\right)_{\infty}}{\left(aq/b, aq^{2}/b, aq^{2}/d, adq; q^{2}\right)_{\infty}}
\end{align}
with $q\to q^2$, $a=q$, $d=q/a$ and $b=q^{1-n}$, we just get the $q$-congruence \eqref{eq;proof10} with $b=q^n$.
\end{proof}
\begin{proof}[Proof of Theorem \ref{thm;4}]
Combining the above two lemmas and the $q$-congruences \eqref{eq1}, \eqref{eq2} with the Chinese remainder theorem for coprime polynomials, we are led to the following result: modulo $[n] (1-aq^{n}) (a-q^{n})  (b-q^{n})$,
\begin{align}
\sum_{k=0}^{(n-1)/2} [6 k+1]  \frac{(a q,q / a ; q^{2})_{k}\left(q ; q^{2}\right)_{k}\left(q/b ; q^{4}\right)_{k} q^{k^{2}+k}b^{k}}{(a q^{4},q^{4} / a ; q^{4})_{k}\left(q^{4} ; q^{4}\right)_{k}\left(q^{2}  b ; q^{2}\right)_{k} }\equiv [n] A_q(a,b,n).\label{eq;proof11}
\end{align}
In the same manner as the proof of Theorem \ref{thm;1}, taking $b\rightarrow  1$ and $a\rightarrow 1$ in sequence, we
immediately arrive at Theorem \ref{thm;4} from \eqref{eq;proof11}.
\end{proof}
\section{The other $q$-analogues of Swisher's supercongruence \eqref{eq:swisher1}}
From Theorem \ref{thm;1}, we can get another five different $q$-analogues of Swisher's supercongruence \eqref{eq:swisher1}.
Fixing $c=d=-1$; $c=-1$, $d\rightarrow\infty$; $c=-1$, $d=0$;
$c\rightarrow\infty$, $d\rightarrow\infty$ and $c\rightarrow0$,
$d\rightarrow0$ respectively in Theorem \ref{thm;1},
we obtain the following results.
\begin{corollary}
Let $n\equiv 1 \pmod4$ be positive integer. Then, modulo $[n]\Phi_n(q)^3$,
\begin{align}
&\sum_{k=0}^{M}[8 k+1] \frac{(q ; q^{4})_k^2(q^2;q^8)_k^2}{(q^{4} ; q^{4})_k^{2}(q^8;q^8)_{k}^2} q^{5k}\nonumber\\
 &\equiv\frac{\left(q^{2 } ; q^{4}\right)_{(n-1) / 4}}{\left(q^{4} ; q^{4}\right)_{(n-1) / 4}}[n] q^{(1-n) / 4}\big\{1+[n]^2B(n,q)\big\}
\sum_{k=0}^{M}\frac{(q^{3};q^4)_k(q;q^4)_k^3}{(q^4;q^4)_k(-q^4;q^4)_k^2(q^{2};q^4)_k}q^{4k};\label{c-1}\\
&\sum_{k=0}^{M}[8 k+1] \frac{\left(q ; q^{4}\right)_{k}^3(q^2;q^8)_{k}}{\left(q^{4} ; q^{4}\right)_{k}^{3}(q^8;q^8)_{k}} q^{2k^2+4k}\nonumber\\ &\equiv\frac{\left(q^{2 } ; q^{4}\right)_{(n-1) / 4}}{\left(q^{4} ; q^{4}\right)_{(n-1) / 4}}[n] q^{(1-n) / 4}\big\{1+[n]^2B(n,q)\big\}
\sum_{k=0}^{M}\frac{(q;q^4)_k^3}{(q^4;q^4)_k(-q^4;q^4)_k(q^{2};q^4)_k}q^{4k};\\
&\sum_{k=0}^{M}[8 k+1] \frac{\left(q ; q^{4}\right)_{k}^3(q^2;q^8)_{k}}{\left(q^{4} ; q^{4}\right)_{k}^{3}(q^8;q^8)_{k}} q^{3k-2k^2}\nonumber\\ &\equiv\frac{\left(q^{2 } ; q^{4}\right)_{(n-1) / 4}}{\left(q^{4} ; q^{4}\right)_{(n-1) / 4}}[n] q^{(1-n) / 4}\big\{1+[n]^2B(n,q)\big\}
\sum_{k=0}^{M}\frac{(q;q^4)_k^3}{(q^4;q^4)_k(-q^4;q^4)_k(q^{2};q^4)_k}(-q)^{3k};\\
&\sum_{k=0}^{M}[8 k+1] \frac{\left(q ; q^{4}\right)_{k}^4}{\left(q^{4} ; q^{4}\right)_{k}^{4}} q^{4k^2+3k}\nonumber\\ &\equiv\frac{\left(q^{2 } ; q^{4}\right)_{(n-1) / 4}}{\left(q^{4} ; q^{4}\right)_{(n-1) / 4}}[n] q^{(1-n) / 4}\big\{1+[n]^2B(n,q)\big\}
\sum_{k=0}^{M}\frac{(q;q^4)_k^3}{(q^4;q^4)_k(q^{2};q^4)_k}q^{4k};\\
&\sum_{k=0}^{M}[8 k+1] \frac{\left(q ; q^{4}\right)_{k}^4}{\left(q^{4} ; q^{4}\right)_{k}^{4}} q^{k-4k^2}\nonumber\\ &\equiv\frac{\left(q^{2 } ; q^{4}\right)_{(n-1) / 4}}{\left(q^{4} ; q^{4}\right)_{(n-1) / 4}}[n] q^{(1-n) / 4}\big\{1+[n]^2B(n,q)\big\}
\sum_{k=0}^{M}\frac{(q;q^4)_k^3}{(q^4;q^4)_k(q^{2};q^4)_k}(-q)^{2k-2k^2}.\label{c-5}
\end{align}
\end{corollary}
Taking $q\rightarrow1$ and $n=p\equiv1 \pmod 4$ be a prime in \eqref{c-1}--\eqref{c-5} and the two $q$-supercongruences \eqref{e1}--\eqref{5} in Theorem \ref{thm;e}, and applying the known congruence by Sun \cite[p.~7]{Sun}:
\[H_{(p-1) / 2}^{(2)} \equiv 0 \pmod p \quad \text { for } \quad p>3,\] we are led to the following supercongruence
\begin{align}
\sum_{k=0}^{m}(8k+1)\frac{(\frac{1}{4})_k^4}{k!^4}
\equiv\frac{\left(\frac{1}{2}\right)_{(p-1) / 4}}{(1)_{(p-1) / 4}} p\left\{1-\frac{p^{2}}{16} H_{(p-1) / 4}^{(2)}\right\} \pmod {p^{4}},
\end{align}
where $m=(p-1)/4$ or $p-1$ and the harmonic numbers of $2$-order is defined as
\[H_{n}^{(2)}=\sum_{k=1}^{n} \frac{1}{k^{2}} \quad \text { with } \quad n \in \mathbb{N}.\]

Recall that for prime $p\ge 5$ and rational number $a\in   \mathbb{Z}_p$, the $p$-adic Gamma function $\Gamma_{p}$ has the following basic properties \cite{Long},
\begin{equation*}
\Gamma_{p}(1)=-1,\quad \Gamma_{p}\left(\frac{1}{2}\right)^{2}=(-1)^{\frac{p+1}{2}}\quad
(a)_n=(-1)^n\frac{\Gamma_{p}(a+n)}{\Gamma_{p}(a)},\quad G_1(a)=G_1(1-a),
\end{equation*}
\begin{equation*}
\Gamma_{p}(a+bp)\equiv \Gamma_{p}(a)\big(1+G_1(a)bp+\frac{G_2(a)b^2p^2}{2}\big) \pmod{p^3},
\end{equation*}
\begin{equation*}
\Gamma_{p}(x) \Gamma_{p}(1-x)=(-1)^{a_{0}(x)},
\end{equation*}
where $G_{k}(a):=\Gamma_{p}^{(k)}(a) / \Gamma_{p}(a)$, $\Gamma_{p}^{(k)}(a)$ is $k$-th derivative of $\Gamma_{p}(a)$ and $a_{0}(x) \in\{1,2, \cdots, p\}$ satisfies $x-a_{0}(x) \equiv 0 \pmod p$.
We have the following result
\begin{align}
\frac{(\frac{1}{2})_{(p-1)/4}}{(1)_{(p-1)/4}}&=\frac{\Gamma_{p}({(1+p)/4})\Gamma_{p}{(1)}}  {\Gamma_{p}{(1/2)}\Gamma_{p}({(3+p)/4})}\nonumber\\
&=-(-1)^{(p-1)/4}\Gamma_{p}{(1/2)}\Gamma_{p}({(1+p)/4})\Gamma_{p}({(1-p)/4})\nonumber\\
&\equiv -(-1)^{(p-1)/4}\Gamma_{p}{(1/2)} \left\{\Gamma_{p}(1 / 4)+\Gamma_{p}^{\prime}(1 / 4) \frac{p}{4}+\Gamma_{p}^{\prime \prime}(1 / 4) \frac{p^{2}}{2 \times 4^{2}}\right\}\nonumber\\
& \times \left\{\Gamma_{p}(1 / 4)-\Gamma_{p}^{\prime}(1 / 4) \frac{p}{4}+\Gamma_{p}^{\prime \prime}(1 / 4) \frac{p^{2}}{2 \times 4^{2}}\right\}\pmod {p^3}.\label{eq;proof12}
\end{align}

Recently, Wang and Pan \cite[Page 6]{WP} have proved that
\begin{align}
H_{(p-1) / 4}^{(2)} \equiv \frac{\Gamma_{p}^{\prime \prime}(1 / 4)}{\Gamma_{p}(1 / 4)}-\left\{\frac{\Gamma_{p}^{\prime}(1 / 4)}{\Gamma_{p}(1 / 4)}\right\}^{2} \pmod p. \label{eq;proof13}
\end{align}
Thus, the combination of \eqref{eq;proof12} and \eqref{eq;proof13} can bring out
\begin{align*}
p\frac{\left(\frac{1}{2}\right)_{(p-1) / 4}}{(1)_{(p-1) / 4}} \left\{1-\frac{p^{2}}{16} H_{(p-1) / 4}^{(2)}\right\} \equiv -(-1)^{(p-1)/4}p\Gamma_{p}{(1/2)}\Gamma_{p}(1/4)^2 \pmod {p^4}.
\end{align*}
This shows that our $q$-congruences \eqref{e1}, \eqref{5} and \eqref{c-1}--\eqref{c-5} are indeed $q$-analogues of Swisher's supercongruence \eqref{eq:swisher1}.

\end{document}